\documentclass[11pt,lenq]{amsart}
\usepackage{amsmath}
\usepackage{amscd}
\usepackage{amssymb}
\usepackage{amsbsy}
\usepackage{amsfonts}
\usepackage{color}
\usepackage{latexsym}
\usepackage{graphics}
\usepackage{graphicx}
\usepackage{amsmath,amscd,latexsym}
\usepackage{multirow}
\usepackage{array}
\usepackage{paralist}
\usepackage{titletoc}

\theoremstyle{plain}
\newtheorem{theorem}{Theorem}[section]
\newtheorem{proposition}[theorem]{Proposition}
\newtheorem{lemma}[theorem]{Lemma}

\newtheorem{remark}[theorem]{Remark}
\newtheorem{definition}[theorem]{Definition}
\newtheorem{notation}[theorem]{Notation}

\newtheorem{main theorem}[theorem]{Main Theorem}

\newlength\savewidth

\newcommand{\ZZ}{\mathbb{Z}}
\newcommand{\QQ}{\mathbb{Q}}
\newcommand{\RR}{\mathbb{R}}

\newcommand{\QQQ}{\hat{\mathbb{Q}}}

\newcommand{\PConway}{\mbox{\boldmath$S$}}

\newcommand{\rtangle}[1]{(B^3,t({#1}))}

\newcommand{\svert}{\,|\,}

\newcommand{\llangle}{\langle\langle}
\newcommand{\rrangle}{\rangle\rangle}

\newcommand{\lp}{(\hskip -0.07cm (}
\newcommand{\rp}{)\hskip -0.07cm )}

\makeatletter
\renewcommand\subsection{\@startsection{subsection}{2}{0mm}
    {-10.5dd plus-8pt minus-4pt}{10.5dd}
     {\normalsize\upshape}}
\makeatother

\begin{document}

\title[A family of two generator non-Hopfian groups]
{A family of two generator non-Hopfian groups}
\author{Donghi Lee}
\address{Department of Mathematics\\
Pusan National University \\
San-30 Jangjeon-Dong, Geumjung-Gu, Pusan, 609-735, Korea}
\email{donghi@pusan.ac.kr}

\author{Makoto Sakuma}
\address{Department of Mathematics\\
Graduate School of Science\\
Hiroshima University\\
Higashi-Hiroshima, 739-8526, Japan}
\email{sakuma@math.sci.hiroshima-u.ac.jp}

\subjclass[2010]{Primary 20F06, 57M25 \\
\indent {The first author was supported by Basic Science Research Program
through the National Research Foundation of Korea(NRF) funded
by the Ministry of Education, Science and Technology(2014R1A1A2054890).
The second author was supported by JSPS Grants-in-Aid 15H03620.}}


\begin{abstract}
We construct $2$-generator non-Hopfian groups
$G_m, m=3, 4, 5, \dots$,
where each $G_m$ has a specific presentation
$G_m=\langle a, b \, | \, u_{r_{m,0}}=u_{r_{m,1}}=u_{r_{m,2}}= \cdots =1 \rangle$
which satisfies small cancellation conditions $C(4)$ and $T(4)$.
Here, $u_{r_{m,i}}$ is the single relator of the upper presentation
of the $2$-bridge link group
of slope $r_{m,i}$,
where $r_{m,0}=[m+1,m,m]$ and
$r_{m,i}=[m+1,m-1,(i-1)\langle m \rangle,m+1,m]$ in continued fraction expansion
for every integer $i \ge 1$.
\end{abstract}

\maketitle


\section{Introduction}
\label{sec:introduction}

Recall that a group $G$ is called {\it Hopfian} if every epimorphism $G \rightarrow G$
is an automorphism.
The non-Hopfian property of finitely generated groups
has a close connection with the non-residual finiteness.
In fact, the classical work due to Mal'cev~\cite{Mal} shows that
every finitely generated non-Hopfian group is
non-residually finite.
One of the hardest open problems about hyperbolic groups is
whether or not every hyperbolic group is residually finite.
An important progress on this problem was given by
Sela~\cite{Sela} asserting that
every torsion-free hyperbolic group is Hopfian.
In 2007, Osin~\cite{Osin} proved that this problem is equivalent to the question
on whether or not a group $G$ is residually finite
if $G$ is hyperbolic relative to a finite collection of residually finite subgroups.
The notion of relatively hyperbolic groups is
an important generalization of hyperbolic groups in geometric group theory
originally introduced by Gromov \cite{Gromov}
(cf. \cite{Bowditch}, \cite{Farb}, \cite{Osin2}).
Motivating examples for this generalization
include the fundamental groups of non-compact hyperbolic manifolds
of finite volume.
In particular, every $2$-bridge link complement
except for a torus link is a hyperbolic manifold with cusps,
so its fundamental group, that is, the $2$-bridge link group,
is hyperbolic relative to its peripheral subgroups
although it is not a hyperbolic group.
It is known by Groves~\cite{Groves} that a finitely generated
torsion-free group is Hopfian,
if it is hyperbolic relative to free abelian subgroups.
It is also proved by Reinfeldt and Weidmann~\cite{Rein, Rein-Weid} that
every hyperbolic group, possibly with torsion,
is Hopfian.
In addition, based on this result,
Coulon and Guirardel~\cite{Coul_Guir}
proved that every lacunary hyperbolic group, which is characterized as
a direct limit of hyperbolic groups with a certain radii condition,
is also Hopfian.

As for small cancellation groups,
it is known that
if a group has a finite presentation which satisfies
small cancellation conditions either $C'(1/6)$ or both $C'(1/4)$ and $T(4)$,
then it is hyperbolic (see~\cite{Strebel}).
Wise~\cite{Wise2} also proved that
every finite $C'(1/6)$-small cancellation presentation defines
a residually finite group.

Historically, not many have been known examples of finitely generated
non-Hopfian groups with specific presentations.
The earliest such example was found by Neumann~\cite{Neum} in 1950
as follows:
$\langle a,b \svert e_2=e_3=\cdots=1 \rangle$,
where $e_i= a^{-1}b^{-1}ab^{-i}ab^{-1}a^{-1}b^ia^{-1}bab^{-i}aba^{-1}b^i$
for every integer $i \ge 2$.
Soon after, the first non-Hopfian group with finite presentation
was discovered by Higman~\cite{Higman} as follows:
$\langle a, s, t \svert a^s=a^2, \ a^t=a^2 \rangle$.
Also a non-Hopfian group with the simplest presentation up to now
was produced by Baumslag and Solitar~\cite{Baum_Soli} as follows:
$\langle a, t \svert (a^2)^t = a^3 \rangle$.
Many other non-Hopfian groups with specific finite presentations
have been obtained by generalizing Higman's group or Baumslag-Solitar's group
(see, for instance, \cite{Sapir_Wise}, \cite{Wise}).
Another notable non-Hopfian group was obtained
by Ivanov and Storozhev~\cite{Ivanov} in 2005.
They constructed a family of finitely generated,
but not finitely presented,
non-Hopfian relatively free groups with direct limits of hyperbolic groups,
although the defining relations of their group presentations
are not explicitly described in terms of generators.

Motivated by this background, we construct non-Hopfian groups
by using hyperbolic $2$-bridge link groups.
In more detail, we construct a family of
$2$-generator non-Hopfian groups
each of which has the form
$G=\langle a, b \, | \, u_{r_0}=u_{r_1}=u_{r_2}=\cdots=1 \rangle$
satisfying small cancellation conditions $C(4)$ and $T(4)$,
where $u_{r_i}$ is the single relator of the upper presentation
of the link group of the $2$-bridge link
of slope $r_i$
for every $i=0,1,2, \dots$.
Here, the rational numbers $r_i$ may be parametrized by $i \ge 0$,
and there is an explicit formula to express $u_{r_i}$ in terms of $a$ and $b$.
To parametrize the rational numbers $r_i$, we express $r_i$
in continued fraction expansion.
Note that every rational number $0<s\le 1$ has a unique continued fraction expansion
such that
\begin{center}
\begin{picture}(230,70)
\put(0,48){$\displaystyle{
s=
\cfrac{1}{m_1+
\cfrac{1}{ \raisebox{-5pt}[0pt][0pt]{$m_2 \, + \, $}
\raisebox{-10pt}[0pt][0pt]{$\, \ddots \ $}
\raisebox{-12pt}[0pt][0pt]{$+ \, \cfrac{1}{m_k}$}
}} \
=:[m_1,m_2, \dots,m_k],}$}
\end{picture}
\end{center}
where $k \ge 1$,
$(m_1, \dots, m_k) \in (\mathbb{Z}_+)^{k}$ and
$m_k \ge 2$ unless $k=1$.

The main result of the present paper is the following,
whose proof is contained in in Section~\ref{sec:proof}.

\begin{theorem}
\label{thm:simplest_case}
Let $r_0=[4,3,3]$, and let
$r_i=[4,2,(i-1)\langle 3 \rangle,4,3]$ for every integer $i \ge 1$.
Then the group presentation
$G=\langle a, b \svert u_{r_0}=u_{r_1}=u_{r_2}=\cdots=1 \rangle$
satisfies small cancellation conditions $C(4)$ and $T(4)$,
and $G$ is non-Hopfian.
\end{theorem}

Here, the symbol ``$(i-1)\langle 3 \rangle$'' represents $i-1$ successive $3$'s if $i-1\ge 1$,
whereas ``$0 \langle 3 \rangle$'' means that $3$ does not occur in that place,
so that $r_1=[4,2,0 \langle 3 \rangle,4,3]=[4,2,4,3]$.

\begin{remark}
{\rm
(1) Once we allow the components of a continued fraction expansion
to be ``$-$'',
meaning that the two integers immediately before and after $-$ are added
to form one component,
$r_i$'s in Theorem~\ref{thm:simplest_case} can be parametrized including $i=0$
as $r_i=[4,i\langle 2,1,- \rangle,3,3]$ for every $i=0,1,2, \dots$.

(2) If we express the rational number $r_i$
in Theorem~\ref{thm:simplest_case} as $q_i/p_i$,
where $p_i$ and $q_i$ are relatively prime positive integers,
then $|u_{r_i}|=2p_i$ (see Section~\ref{subsec:presentation}).
A simple computation shows that
the inequality $3< p_{i+1}/p_i <4$ holds for every $i=0,1,2,\dots$,
so that
the length $|u_{r_i}|$ of the word $u_{r_i}$ satisfies the inequality
$c \cdot 3^i < |u_{r_i}| < c \cdot 4^{i}$
for every integer $i \ge 1$, where $c=|u_{r_0}|$.
}
\end{remark}

By looking at the proof of Theorem~\ref{thm:simplest_case}
in Section~\ref{sec:proof},
it is not hard to see that a similar result holds not only for $r_0=[4,3,3]$
but also for $r_0=[m+1,m,m]$ with $m$ being any integer
greater than $3$.
Thus we only state its general form without a detailed proof.

\begin{theorem}
\label{thm:general_case}
Suppose that $m$ is an integer with $m \ge 3$.
Let $r_0=[m+1,m,m]$, and let
$r_i=[m+1,m-1,(i-1)\langle m \rangle,m+1,m]$ for every integer $i \ge 1$.
Then the group presentation
$G=\langle a, b \svert u_{r_0}=u_{r_1}=u_{r_2}=\cdots=1 \rangle$
satisfies small cancellation conditions $C(4)$ and $T(4)$,
and $G$ is non-Hopfian.
\end{theorem}

The present paper is organized as follows.
In Section~\ref{sec:preliminaries}, we recall
the upper presentation of a $2$-bridge link group, and basic facts
established in \cite{lee_sakuma} concerning the upper presentations.
We also recall key facts from \cite{lee_sakuma} obtained by applying small cancellation theory to
the upper presentations.
Section~\ref{sec:proof} is devoted to the proof of the main result
(Theorem~\ref{thm:simplest_case}).

\section{Preliminaries}
\label{sec:preliminaries}

\subsection{Upper presentations of $2$-bridge link groups}
\label{subsec:presentation}

We recall some notation in \cite{lee_sakuma}.
The {\it Conway sphere} $\PConway$ is the 4-times punctured sphere
which is obtained as the quotient of $\RR^2-\ZZ^2$
by the group generated by the $\pi$-rotations around
the points in $\ZZ^2$.
For each $s \in \QQQ:=\QQ\cup\{\infty\}$,
let $\alpha_s$ be the simple loop in $\PConway$
obtained as the projection of a line in $\RR^2-\ZZ^2$
of slope $s$.
We call $s$ the {\it slope} of the simple loop $\alpha_s$.

For each $r\in \QQQ$,
the {\it $2$-bridge link $K(r)$ of slope $r$}
is the sum of the rational tangle
$\rtangle{\infty}$ of slope $\infty$ and
the rational tangle $\rtangle{r}$ of slope $r$.
Recall that $\partial(B^3-t(\infty))$ and $\partial(B^3-t(r))$
are identified with $\PConway$
so that $\alpha_{\infty}$ and $\alpha_r$
bound disks in $B^3-t(\infty)$ and $B^3-t(r)$, respectively.
By van-Kampen's theorem, the link group $G(K(r))=\pi_1(S^3-K(r))$ is obtained as follows:
\[
G(K(r))=\pi_1(S^3-K(r))
\cong \pi_1(\PConway)/ \llangle\alpha_{\infty},\alpha_r\rrangle
\cong \pi_1(B^3-t(\infty))/\llangle\alpha_r\rrangle.
\]
Let $\{a,b\}$ be the standard meridian generator pair of $\pi_1(B^3-t(\infty), x_0)$
as described in \cite[Section~3]{lee_sakuma}.
Then $\pi_1(B^3-t(\infty))$ is identified with
the free group $F(a,b)$ with basis $\{a, b\}$.
For the rational number $r=q/p$, where $p$ and $q$ are relatively prime positive integers,
let $u_r$ be the word in $\{a,b\}$ obtained as follows.
Set $\epsilon_i = (-1)^{\lfloor iq/p \rfloor}$,
where $\lfloor x \rfloor$ is the greatest integer not exceeding $x$.
\begin{enumerate}[\indent \rm (1)]
\item If $p$ is odd, then
\[u_{q/p}=a\hat{u}_{q/p}b^{(-1)^q}\hat{u}_{q/p}^{-1},\]
where
$\hat{u}_{q/p} = b^{\epsilon_1} a^{\epsilon_2} \cdots b^{\epsilon_{p-2}} a^{\epsilon_{p-1}}$.
\item If $p$ is even, then
\[u_{q/p}=a\hat{u}_{q/p}a^{-1}\hat{u}_{q/p}^{-1},\]
where
$\hat{u}_{q/p} = b^{\epsilon_1} a^{\epsilon_2} \cdots a^{\epsilon_{p-2}} b^{\epsilon_{p-1}}$.
\end{enumerate}
Then $u_r\in F(a,b)\cong\pi_1(B^3-t(\infty))$
is represented by the simple loop $\alpha_r$,
and we obtain the following two-generator and one-relator presentation
of a $2$-bridge link groups:
\[
G(K(r)) \cong\pi_1(B^3-t(\infty))/\llangle \alpha_r\rrangle
\cong \langle a, b \svert u_r \rangle.
\]
This presentation is called the {\it upper presentation} of a $2$-bridge link group.

\subsection{Basic facts concerning the upper presentations}

Throughout this paper,
a {\it cyclic word} is defined to be the set of all cyclic permutations of a
cyclically reduced word. By $(v)$ we denote the cyclic word associated with a
cyclically reduced word $v$.
Also the symbol ``$\equiv$'' denotes the {\it letter-by-letter equality}
between two words or between two cyclic words.
Now we recall definitions and basic facts from \cite{lee_sakuma}
which are needed in the proof of Theorem~\ref{thm:simplest_case}
in Section~\ref{sec:proof}.

\begin{definition}
\label{def:alternating}
{\rm (1) Let $v$ be a reduced word in
$\{a,b\}$. Decompose $v$ into
\[
v \equiv v_1 v_2 \cdots v_t,
\]
where, for each $i=1, \dots, t-1$, $v_i$ is a positive (resp., negative) subword
(that is, all letters in $v_i$ have positive (resp., negative) exponents),
and $v_{i+1}$ is a negative (resp., positive) subword.
Then the sequence of positive integers
$S(v):=(|v_1|, |v_2|, \dots, |v_t|)$ is called the {\it $S$-sequence of $v$}.

(2) Let $v$ be a cyclically reduced word in
$\{a, b\}$. Decompose the cyclic word $(v)$ into
\[
(v) \equiv (v_1 v_2 \cdots v_t),
\]
where $v_i$ is a positive (resp., negative) subword,
and $v_{i+1}$ is a negative (resp., positive) subword (taking
subindices modulo $t$). Then the {\it cyclic} sequence of positive integers
$CS(v):=\lp |v_1|, |v_2|, \dots, |v_t| \rp$ is called
the {\it $CS$-sequence of $(v)$}.
Here the double parentheses denote that the sequence is considered modulo
cyclic permutations.
}
\end{definition}

\begin{definition}
\label{def4.1(3)}
{\rm
For a rational number $r$ with $0<r\le 1$,
let $u_r$ be the word defined in the beginning of this section.
Then the symbol $CS(r)$ denotes the
$CS$-sequence $CS(u_r)$ of $(u_r)$, which is called
the {\it $CS$-sequence of slope $r$}.}
\end{definition}

A reduced word $w$ in $\{a,b\}$ is said to be {\it alternating}
if $a^{\pm 1}$ and $b^{\pm 1}$ appear in $w$ alternately,
to be precise, neither $a^{\pm2}$ nor $b^{\pm2}$ appears in $w$.
Also a cyclically reduced word $w$ in $\{a,b\}$
is said to be {\it cyclically alternating},
i.e., all the cyclic permutations of $w$ are alternating.
In particular, $u_r$ is a cyclically alternating word in $\{a,b\}$.
Note that every alternating word $w$ in $\{a,b\}$
is determined by the sequence $S(w)$
and the initial letter (with exponent) of $w$.
Note also that if $w$ is a cyclically alternating word in $\{a,b\}$
such that $CS(w)=CS(r)$, then either $(w) \equiv (u_r)$
or $(w) \equiv (u_r^{-1})$ as cyclic words.

In the remainder of this section, we suppose that $r$ is a rational number
with $0<r\le1$, and write $r$ as a continued fraction expansion
$r=[m_1,m_2, \dots,m_k]$,
where $k \ge 1$, $(m_1, \dots, m_k) \in (\mathbb{Z}_+)^k$ and
$m_k \ge 2$ unless $k=1$.
Note from \cite{lee_sakuma} that
if $k \ge 2$, then some properties of $CS(r)$
differ according to $m_2=1$ or $m_2 \ge 2$.
For brevity, we write $m$ for $m_1$.

\begin{lemma} [{\cite[Proposition~4.3]{lee_sakuma}}]
\label{lem:properties}
For the rational number $r=[m_1,m_2, \dots,m_k]$
satisfying that $m_2 \ge 2$ if $k \ge 2$,
the following hold.
\begin{enumerate}[\indent \rm (1)]
\item Suppose $k=1$, i.e., $r=1/m$.
Then $CS(r)=\lp m,m \rp$.

\item Suppose $k\ge 2$. Then each term of $CS(r)$ is either $m$ or $m+1$.
Moreover, no two consecutive terms of $CS(r)$ can be $(m+1, m+1)$,
so there is a cyclic sequence of positive integers $\lp t_1,t_2,\dots,t_s \rp$ such that
\[
CS(r)=\lp m+1, t_1\langle m\rangle, m+1, t_2\langle m\rangle,
\dots,m+1, t_s\langle m\rangle \rp.
\]
Here, the symbol ``$t_i\langle m\rangle$'' represents $t_i$ successive $m$'s.
\end{enumerate}
\end{lemma}

\begin{definition}
\label{def_CT(r)}
{\rm
If $k\ge 2$, the symbol $CT(r)$ denotes the cyclic sequence
$\lp t_1,t_2,\dots,t_s \rp$ in Lemma~\ref{lem:properties},
which is called the {\it $CT$-sequence of slope $r$}.
}
\end{definition}

\begin{lemma} [{\cite[Proposition~4.4 and Corollary~4.6]{lee_sakuma}}]
\label{lem:induction1}
For the rational number $r=[m_1,m_2, \dots,m_k]$ with $k\ge 2$ and $m_2\ge 2$,
let $r'$ be the rational number defined as
\[
r'=[m_2-1, m_3, \dots, m_k].
\]
Then we have
$CT(r)=CS(r')$.
\end{lemma}

\begin{lemma} [{\cite[Proposition~4.5]{lee_sakuma}}]
\label{lem:sequence}
For the rational number $r=[m_1,m_2, \dots,m_k]$,
the cyclic sequence $CS(r)$ has a decomposition
$\lp S_1, S_2, S_1, S_2 \rp$ which satisfies the following.
\begin{enumerate} [\indent \rm (1)]
\item Each $S_i$ is symmetric,
i.e., the sequence obtained from $S_i$ by reversing the order is
equal to $S_i$. {\rm (}Here, $S_1$ is empty if $k=1$.{\rm )}

\item Each $S_i$ occurs only twice in
the cyclic sequence $CS(r)$.

\item The subsequence $S_1$ begins and ends with $m+1$.

\item The subsequence $S_2$ begins and ends with $m$.
\end{enumerate}
\end{lemma}

\begin{lemma} [{\cite[Proof of Proposition~4.5]{lee_sakuma}}]
\label{lem:relation}
For the rational number $r=[m_1,m_2, \dots,m_k]$ with $k\ge 2$ and $m_2\ge 2$,
let $r'$ be the rational number defined as in Lemma~\ref{lem:induction1}.
Also let $CS(r')=\lp T_1, T_2, T_1, T_2 \rp$ and $CS(r)=\lp S_1, S_2, S_1, S_2 \rp$
be the decompositions described in Lemma~\ref{lem:sequence}.
Then the following hold.
\begin{enumerate} [\indent \rm (1)]
\item If $k=2$, then $T_1=\emptyset$, $T_2=(m_2-1)$,
and $S_1 =(m+1)$, $S_2 =((m_2-1)\langle m \rangle)$.

\item If $k\ge 3$, then $T_1=(t_1, \dots, t_{s_1})$,
$T_2=(t_{s_1+1}, \dots, t_{s_2})$, and
\[
\begin{aligned}
S_1 &=(m+1, t_{s_1+1}\langle m\rangle, m+1,
\dots, m+1, t_{s_2}\langle m\rangle, m+1),\\
S_2 &=(t_1\langle m\rangle, m+1,t_2\langle m\rangle, \dots,
t_{s_1-1}\langle m\rangle, m+1, t_{s_1}\langle m\rangle).
\end{aligned}
\]
\end{enumerate}
\end{lemma}

The following lemma is useful in the proof of Lemma~\ref{lem:claim5}.

\begin{lemma}
\label{lem:useful}
For two distinct rational numbers
$r=[m_1, m_2, \dots, m_k]$ and $s=[m_1, l_2, \dots, l_t]$,
assume that
\begin{enumerate} [\indent \rm (i)]
\item $m$ is a positive integer;

\item $m_i$ and $l_j$ are integers greater than $1$ for every $i\ge 2$ and $j \ge 2$;

\item $k, t \ge 3$ and $k \neq t$; and

\item if $k<t$, then $m_2 \ge l_2$, while if $k>t$, then $m_2 \le l_2$.
\end{enumerate}
Let $r'$ and $s'$ be the rational numbers defined as in Lemma~\ref{lem:induction1}.
Also let $CS(r)=\lp S_1, S_2, S_1, S_2 \rp$ and $CS(r')=\lp T_1, T_2, T_1, T_2 \rp$
be the decompositions described in Lemma~\ref{lem:sequence}.
Suppose that $CS(s)$ contains $S_1$ or $S_2$ as a subsequence.
Then $CS(s')$ contains $T_1$ or $T_2$ as a subsequence.
\end{lemma}

In the above lemma (and throughout this paper),
we mean by a {\it subsequence}
a subsequence without leap.
Namely a sequence $(a_1,a_2,\dots, a_l)$
is called a {\it subsequence} of a cyclic sequence,
if there is a sequence $(b_1,b_2,\dots, b_n)$
representing the cyclic sequence
such that $l \le n$ and
$a_i=b_i$ for $1\le i\le l$.

\begin{proof}
First suppose that $CS(s)$ contains $S_1$ as a subsequence.
By Lemma~\ref{lem:relation}(2),
$CS(s)$ contains $(m+1, t_{s_1+1}\langle m\rangle, m+1,
\dots, m+1, t_{s_2}\langle m\rangle, m+1)$ as a subsequence,
where $T_2=(t_{s_1+1}, \dots, t_{s_2})$.
Then clearly $CS(s')=CT(s)$ contains
$(t_{s_1+1}, \dots, t_{s_2})$, that is, $T_2$,
as a subsequence. So we are done.

Next suppose that $CS(s)$ contains $S_2$ as a subsequence.
Again by Lemma~\ref{lem:relation}(2),
$CS(s)$ contains $(t_1\langle m\rangle, m+1,t_2\langle m\rangle, \dots,
t_{s_1-1}\langle m\rangle, m+1, t_{s_1}\langle m\rangle)$ as a subsequence,
where $T_1=(t_1, \dots, t_{s_1})$.
Then $CS(s')=CT(s)$ contains
$(d_1+t_1, t_2, \dots, t_{s_1-1}, t_{s_1}+d_2)$
as a subsequence, where $d_1, d_2 \ge 0$.
In the reminder of the proof, we show that $d_1=d_2=0$,
so that $CS(s')=CT(s)$ contains
$T_1=(t_1, t_2, \dots, t_{s_1})$ as a subsequence.
To this end, note that since $r'=[m_2-1, m_3, \dots, m_k]$,
$t_1=t_{s_1}=m_2$ by Lemma~\ref{lem:sequence}(3).
Also since $s'=[l_2-1, l_3, \dots, l_t]$,
$CS(s')=CT(s)$ consists of $l_2-1$ and $l_2$ by Lemma~\ref{lem:properties}(2).
Hence each of
$d_1+t_1=d_1+m_2$ and $t_{s_1}+d_2=m_2+d_2$ is either $l_2-1$ or $l_2$.
Suppose first that $k<t$.
Then $m_2 \ge l_2$ by the assumption (iv),
and thus the only possibility is $m_2=l_2$. Thus we have $d_1=d_2=0$.
Suppose next that $k>t$.
Then $m_2 \le l_2$ again by the assumption (iv).
Note that $k\ge 4$, because $t \ge 3$ by the assumption (iii).
Thus we can see, by using Lemma~\ref{lem:relation}
and the assumption $m_i \ge 2$ for every $i \ge 2$, that
$S_2$ contains $(m+1, (m_2-1) \langle m \rangle, m+1)$ as
a subsequence.
This implies that
$CS(s')=CT(s)$ contains a term $m_2-1$.
Since $m_2-1 \le l_2-1$, the only possibility $m_2=l_2$.
Thus we again have $d_1=d_2=0$, completing
the proof of Lemma~\ref{lem:useful}.
\end{proof}

\subsection{Small cancellation theory applied to the upper presentations}

A subset $R$ of the free group $F(a,b)$ is called {\it symmetrized},
if all elements of $R$ are cyclically reduced and,
for each $w \in R$, all cyclic permutations of $w$ and $w^{-1}$ also belong to $R$.

\begin{definition}
{\rm Suppose that $R$ is
a symmetrized subset of $F(a,b)$.
A nonempty word $v$ is called
a {\it piece {\rm (}with respect to $R$\rm{)}}
if there exist distinct
$w_1, w_2 \in R$ such that $w_1 \equiv vc_1$ and $w_2 \equiv vc_2$.
The small cancellation conditions $C(p)$ and $T(q)$,
where $p$ and $q$ are integers such that $p \ge 2$ and $q \ge 3$,
are defined as follows (see \cite{lyndon_schupp}).
\begin{enumerate}
\item Condition $C(p)$: If $w \in R$
is a product of $n$ pieces, then $n \ge p$.

\item Condition $T(q)$: For
$w_1, \dots, w_n \in R$
with no successive elements $w_i, w_{i+1}$
an inverse pair $(i$ mod $n)$, if $n < q$, then at least one of the products
$w_1 w_2,\dots,$ $w_{n-1} w_n$, $w_n w_1$
is freely reduced without cancellation.
\end{enumerate}
}
\end{definition}

The following proposition enables us to apply small cancellation theory
to the upper presentation
$\langle a, b \, |\, u_r \rangle$ of $G(K(r))$.

\begin{proposition}[{\cite[Theorem 5.1]{lee_sakuma}}]
\label{prop:small_cancellation}
Let $r$ be a rational number such that $0 < r< 1$, and
let $R$ be the symmetrized subset of $F(a, b)$ generated
by the single relator $u_r$
of the group presentation $G(K(r))=\langle a, b \, |\, u_r \rangle$.
Then $R$ satisfies $C(4)$ and $T(4)$.
\end{proposition}

This proposition follows from
the following characterization of pieces,
which in turn is proved by using Lemma~\ref{lem:sequence}.

\begin{lemma}
 [{\cite[Corollary 3.25]{lee_sakuma_2}}]
\label{lem:max_piece}
Let $r$ and $R$ be as in Proposition~\ref{prop:small_cancellation}.
Then a subword $w$ of the cyclic word
$(u_r^{\pm 1})$ is a piece
with respect to $R$
if and only if $S(w)$ contains neither $S_1$
nor $(\ell_1, S_2, \ell_2)$ with $\ell_1,\ell_2\in\ZZ_+$
as a subsequence.
\end{lemma}

\section{Proof of Theorem~\ref{thm:simplest_case}}
\label{sec:proof}

In this section, for brevity of notation, we sometimes
write $\bar{x}$ for $x^{-1}$ for a letter or a word $x$.
For a quotient group $H$ of the free group $F(a,b)$
and two elements $w_1$ and $w_2$ of $F(a,b)$,
the symbol $w_1=_H w_2$ means the equality in the group $H$.

For $r_0=[4,3,3]$, we have
by using Lemma~\ref{lem:relation}
\[
CS(u_{r_0})=CS(r_0)=\lp 5,4,4,4,5,4,4,5,4,4,5,4,4,4,5,4,4,5,4,4\rp.
\]
Let $G_0=\langle a, b \svert u_{r_0}=1 \rangle$.
Also let
$X \equiv a \cdots \bar{a}$ be
the alternating word in $\{a,b\}$
such that $S(X)=(4,4,4,5,4,4)$,
and let $f: F(a,b) \rightarrow F(a,b)$ be the homomorphism
defined by $f(a)=\bar{X}$ and $f(b)=\bar{b}$.

\begin{lemma}
\label{lem:claim1}
Under the foregoing notation,
let $\tilde{f}: F(a,b) \rightarrow G_0$ be the composition of $f$
and the canonical surjection $F(a,b) \rightarrow G_0$.
Then $\tilde{f}$ is onto.
\end{lemma}

\begin{proof}
Since $\tilde{f}(b)=\bar{b}$,
it suffices to show that
$a\in G_0$
is contained in the image of $\tilde{f}$.
Let $w \equiv a \cdots \bar{a}$ be
the alternating word in $\{a,b\}$
such that $S(w)=(3,3,3,4,3,3)$.
Then
\[
f(w)=\bar{X}\bar{b} \bar{X} bX b\bar{X}\bar{b} \bar{X} bX bX\bar{b} \bar{X}\bar{b} X bX.
\]
Here, since $X \equiv a \cdots \bar{a}$ and $\bar{X} \equiv a \cdots \bar{a}$
are alternating words in $\{a,b\}$,
we see that $f(w) \equiv a \cdots \bar{a}$ is also
an alternating word in $\{a,b\}$ with
\[
S(f(w))=(S(\bar{X}\bar{b}), S(\bar{X}), S(bX), S(b\bar{X}\bar{b}),
S(\bar{X}), S(bX), S(bX\bar{b}), S(\bar{X}\bar{b}), S(X), S(bX)).
\]
Since $S(X)=(4,4,4,5,4,4)$ and $S(\bar{X})=(4,4,5,4,4,4)$,
we have $S(\bar{X}\bar{b})=(4,4,5,4,4,5)$, $S(bX)=(5,4,4,5,4,4)$,
$S(b\bar{X}\bar{b})=(5,4,5,4,4,5)$ and $S(bX\bar{b})=(5,4,4,5,4,5)$, so that
\[
\begin{aligned}
S(f(w))&=(4,4,5,4,4,5, 4,4,5,4,4,4, 5,4,4,5,4,4, 5,4,5,4,4,5,
4,4,5,4,4,4, \\
& \qquad 5,4,4,5,4,4, 5,4,4,5,4,5, 4,4,5,4,4,5, 4,4,4,5,4,4, 5,4,4,5,4,4).
\end{aligned}
\]
Letting $v_1 \equiv a \cdots \bar{b}$,
$v_2 \equiv a \cdots \bar{b}$,
and $v_3 \equiv b \cdots \bar{a}$ be
the cyclically alternating words in $\{a,b\}$ such that
\[
\begin{aligned}
S(v_1)&=(4,4,5,4,4,5, 4,4,5,4,4,4, 5,4,4,5,4,4,5,4),\\
S(v_2)&=(5,4,4,5,4,4,5,4,4,4,5,4,4,5,4,4, 5,4,4,4),\\
S(v_3)&=(4,5, 4,4,5,4,4,5, 4,4,4,5,4,4, 5,4,4,5,4,4),
\end{aligned}
\]
we see that $f(w) \equiv v_1 v_2 \bar{a} v_3$.
Moreover, for each $i=1,2,3$, since $CS(v_i)=CS(r_0)$,
$(v_i) \equiv (u_{r_0}^{\pm 1})$ as cyclic words
by \cite[Lemma~3.2]{lee_sakuma},
which implies that $v_i=_{G_0} 1$.
Hence $f(w)=_{G_0} \bar{a}$,
and thus
$a\in G_0$ is contained in the image of $\tilde{f}$, as required.
\end{proof}

At this point, we set up the following notation which will be used
at the end of the proofs of Lemmas~\ref{lem:claim2} and \ref{lem:claim3}.

\begin{notation}
\label{not:TVsequences}
{\rm
(1) Suppose that $v$ is an alternating word in
$\{a,b\}$ such that there is a sequence $(t_1,t_2,\dots,t_s)$ of positive integers satisfying
\[
S(v)=(\epsilon_1\langle 5 \rangle, t_1\langle 4\rangle, 5, t_2\langle 4\rangle,
\dots,5, t_s\langle 4\rangle, \epsilon_2\langle 5 \rangle),
\]
where $\epsilon_i$ is $0$ or $1$ for $i=1,2$.
Then the symbol $T(v)$ denotes the sequence
$(t_1,t_2,\dots,t_s)$.

(2) Suppose that $v$ is a cyclically alternating word in
$\{a,b\}$ such that there is a cyclic sequence
$\lp t_1,t_2,\dots,t_s \rp$ of positive integers satisfying
\[
CS(v)=\lp 5, t_1\langle 4\rangle, 5, t_2\langle 4\rangle,
\dots,5, t_s\langle 4\rangle \rp.
\]
Then the symbol $CT(v)$ denotes the cyclic sequence
$\lp t_1,t_2,\dots,t_s \rp$.
In particular,
by Lemma~\ref{lem:relation},
if $v \equiv u_r$ for some
$r=[4, m_2, \dots, m_k]$ with $m_2 \ge 2$,
then
$CT(u_r)=CT(r)=CS(r')$, where $r'=[m_2-1, \dots, m_k]$.

(3) Suppose that $v$ is an alternating word in
$\{a,b\}$ such that there is a sequence $(h_1,h_2,\dots,h_p)$ of positive integers satisfying
\[
T(v)=(\epsilon_1\langle 2 \rangle, h_1\langle 1\rangle, 2, h_2\langle 1\rangle,
\dots,2, h_p\langle 1\rangle, \epsilon_2\langle 2 \rangle),
\]
where $T(v)$ is defined as in (1) and $\epsilon_i$ is $0$ or $1$ for $i=1,2$.
Then the symbol $V(v)$ denotes the sequence
$(h_1,h_2,\dots,h_p)$.

(4) Suppose that $v$ is a cyclically alternating word in
$\{a,b\}$ such that there is a cyclic sequence $\lp h_1,h_2,\dots,h_p \rp$
of positive integers satisfying
\[
CT(v)=\lp 2, h_1\langle 1\rangle, 2, h_2\langle 1\rangle,
\dots,2, h_p\langle 1\rangle \rp,
\]
where $CT(v)$ is defined as in (2).
Then the symbol $CV(v)$ denotes the cyclic sequence
$\lp h_1,h_2,\dots,h_p \rp$.
In particular,
by Lemma~\ref{lem:relation},
if $v \equiv u_r$ for some
$r=[4, 2, m_3, \dots, m_k]$ with $m_3 \ge 2$,
then
$CV(u_r)=CT(r')=CS(r'')$, where $r'=[1, m_3, \dots, m_k]$
and $r''=[m_3-1, \dots, m_k]$.
}
\end{notation}

\begin{lemma}
\label{lem:claim2}
Under the foregoing notation,
$(f(u_{r_0}))=_{G_0} (u_{r_1}^{\pm 1})$.
\end{lemma}

\begin{proof}
Recall that
\[
CS(u_{r_0})=CS(r_0)=\lp 5,4,4,4,5,4,4,5,4,4,5,4,4,4,5,4,4,5,4,4\rp.
\]
Clearly the cyclic word $(u_{r_0})$ has six positive or negative subwords of length $5$.
Cutting in the middle of such subwords,
we may write the cyclic word $(u_{r_0})$ as a product
$(v_1 \cdots v_6)$, where
\[
\begin{aligned}
v_1& \equiv aba \bar{b}\bar{a}\bar{b}\bar{a} baba \bar{b}\bar{a}\bar{b}\bar{a} b,\\
v_2& \equiv abab \bar{a}\bar{b}\bar{a}\bar{b} abab \bar{a}\bar{b}\bar{a}, \\
v_3&\equiv \bar{b}\bar{a} baba \bar{b}\bar{a}\bar{b}\bar{a} b, \\
v_4&\equiv abab \bar{a}\bar{b}\bar{a}\bar{b} abab \bar{a}\bar{b}\bar{a}\bar{b} ab,\\
v_5&\equiv v_2^{-1}\equiv aba \bar{b}\bar{a}\bar{b}\bar{a} baba \bar{b}\bar{a}\bar{b}\bar{a},\\
v_6&\equiv v_3^{-1}\equiv \bar{b} abab \bar{a}\bar{b}\bar{a}\bar{b} ab.
\end{aligned}
\]

Put
$w_n:\equiv f(v_n)$ for every $n=1, \dots, 6$, namely
\[
\begin{aligned}
w_1& :\equiv \bar{X}\bar{b}\bar{X}bXbX\bar{b}\bar{X}\bar{b}\bar{X}bXbX\bar{b},\\
w_2& :\equiv \bar{X}\bar{b}\bar{X} \bar{b}XbXb\bar{X}\bar{b}\bar{X}\bar{b}XbX, \\
w_3& :\equiv bX\bar{b}\bar{X}\bar{b}\bar{X}bXbX \bar{b}, \\
w_4& :\equiv \bar{X}\bar{b}\bar{X}\bar{b}XbXb\bar{X}\bar{b}\bar{X}\bar{b}XbXb\bar{X} \bar{b},\\
w_5& :\equiv w_2^{-1} \quad \textrm{and} \quad w_6 :\equiv w_3^{-1}.
\end{aligned}
\]
It then follows that
\[
(f(u_{r_0}))=(f(v_1 \cdots v_6))=(w_1 \cdots w_6).
\]

\medskip
\noindent{\bf Claim 1.}
{\it
$\bar{X}\bar{b}\bar{X}bXbX\bar{b}=_{G_0} z_1$,
where $z_1 \equiv a \cdots \bar{b}$ is an alternating word in $\{a,b\}$
with $S(z_1)=(4,5,4,5)$.
}

\begin{proof}[Proof of Claim 1]
Recall that $X \equiv a \cdots \bar{a}$ and $\bar{X} \equiv a \cdots \bar{a}$
are alternating words in $\{a,b\}$
such that $S(X)=(4,4,4,5,4,4)$ and $S(\bar{X})=(4,4,5,4,4,4)$.
It is not hard to see that
\[
\begin{aligned}
S(\bar{X}\bar{b}\bar{X}bXbX\bar{b})&=(S(\bar{X}\bar{b}), S(\bar{X}), S(bX), S(bX\bar{b}))\\
&=((4,4,5,4,4,5),(4,4,5,4,4,4),(5,4,4,5,4,4),(5,4,4,5,4,5))\\
&=(4,4,5,4,4,5,4,4,5,4,4,4,5,4,4,5,4,4,5,4,4,5,4,5).
\end{aligned}
\]
Letting $y_1 \equiv a \cdots \bar{b}$ and $z_1 \equiv a \cdots \bar{b}$
be alternating words in $\{a,b\}$ such that
$S(y_1)=(4,4,5,4,4,5,4,4,5,4,4,4,5,4,4,5,4,4,5,4)$
and $S(z_1)=(4,5,4,5)$,
clearly
$\bar{X}\bar{b}\bar{X}bXbX\bar{b} \equiv y_1z_1$.
Here, since $CS(y_1)=CS(r_0)$ and so $y_1 =_{G_0} 1$,
we finally have
$\bar{X}\bar{b}\bar{X}bXbX\bar{b} \equiv y_1z_1 =_{G_0} z_1$,
as required.
\end{proof}

\medskip
\noindent{\bf Claim 2.}
{\it
$\bar{X}\bar{b}\bar{X} \bar{b}XbX=_{G_0} z_2$,
where $z_2 \equiv a \cdots \bar{a}$ is
the alternating word in $\{a,b\}$ with $S(z_2)=(4,4,5,4)$.
}

\begin{proof}[Proof of Claim 2]
As in the proof of Claim~1, we have
\[
\begin{aligned}
S(\bar{X}\bar{b}\bar{X} \bar{b}XbX)&=(S(\bar{X}\bar{b}), S(\bar{X} \bar{b}), S(X), S(bX))\\
&=((4,4,5,4,4,5), (4,4,5,4,4,5), (4,4,4,5,4,4), (5,4,4,5,4,4))\\
&=(4,4,5,4,4,5,4,4,5,4,4,5,4,4,4,5,4,4,5,4,4,5,4,4).
\end{aligned}
\]
Letting $z_2 \equiv a \cdots \bar{a}$ and
$y_2 \equiv b \cdots \bar{a}$ be alternating words in $\{a,b\}$
such that $S(z_2)=(4,4,5,4)$ and $S(y_2)=(4,5,4,4,5,4,4,5,4,4,4,5,4,4,5,4,4,5,4,4)$,
clearly
$\bar{X}\bar{b}\bar{X} \bar{b}XbX \equiv z_2y_2$.
Here, since $CS(y_2)=CS(r_0)$ and so $y_2 =_{G_0} 1$,
we finally have
$\bar{X}\bar{b}\bar{X} \bar{b}XbX \equiv z_2y_2 =_{G_0} z_2$,
as required.
\end{proof}

By Claims~1 and 2, it follows that
\[
\begin{aligned}
w_1&=_{G_0} z_1^2 \equiv : w_1', \\
w_2&=_{G_0} z_2 z_1^{-1} \equiv : w_2', \\
w_3&=_{G_0} bX\bar{b} z_1 \equiv : w_3', \\
w_4&=_{G_0} z_2 z_1^{-1} b\bar{X} \bar{b} \equiv : w_4',\\
w_5& = w_2^{-1} =_{G_0} (w_2')^{-1} \equiv : w_5', \\
w_6& = w_3^{-1} =_{G_0} (w_3')^{-1} \equiv : w_6',
\end{aligned}
\]
so that
\[
(f(u_{r_0}))=(w_1 \cdots w_6)=_{G_0} (w_1' \cdots w_6').
\]
Moreover, we see that
$w_1' \equiv a \cdots \bar{b}$, $w_2' \equiv a \cdots \bar{a}$,
$w_3' \equiv b \cdots \bar{b}$, $w_4' \equiv a \cdots \bar{b}$,
$w_5' \equiv a \cdots \bar{a}$ and $w_6' \equiv b \cdots \bar{b}$
are alternating words in $\{a,b\}$ such that
\[
\begin{aligned}
S(w_1')&=(S(z_1), S(z_1))=(4,5,4,5,4,5,4,5),\\
S(w_2')&=(S(z_2), S(z_1^{-1}))=(4,4,5,4,5,4,5,4),\\
S(w_3')&=(S(bX\bar{b}), S(z_1))=(5,4,4,5,4,5,4,5,4,5),\\
S(w_4')&=(S(z_2), S(z_1^{-1}),S(b\bar{X} \bar{b}))=(4,4,5,4,5,4,5,4,5,4,5,4,4,5),\\
S(w_5')&=S((w_2')^{-1})=(4,5,4,5,4,5,4,4),\\
S(w_6')&=S((w_3')^{-1})=(5,4,5,4,5,4,5,4,4,5).
\end{aligned}
\]
This implies that
\[
CS(w_1' \cdots w_6')=\lp S(w_1'), \dots, S(w_6') \rp.
\]
Following Notation~\ref{not:TVsequences}, we also have
\[
\begin{aligned}
T(w_1')&=(1,1,1,1), \quad T(w_2')=(2,1,1,1),\\
T(w_3')&=(2,1,1,1), \quad T(w_4')=(2,1,1,1,1,2),\\
T(w_5')&=(1,1,1,2), \quad T(w_6')=(1,1,1,2),
\end{aligned}
\]
and that
\[
CT(w_1' \cdots w_6')=\lp T(w_1'), \dots, T(w_6') \rp.
\]
We furthermore have
\[
\begin{aligned}
V(w_1')&=(4), \quad V(w_2')=(3), \quad V(w_3')=(3),\\
V(w_4')&=(4), \quad V(w_5')=(3), \quad V(w_6')=(3),
\end{aligned}
\]
and
\[
CV(w_1' \cdots w_6')=\lp V(w_1'), \dots, V(w_6') \rp =\lp 4,3,3,4,3,3 \rp.
\]

Since $\lp 4,3,3,4,3,3 \rp$ is the $CS$-sequence corresponding to the
rational number $[3,3]$,
we see that
\[
(w_1' \cdots w_6') \equiv (u_{r}^{\pm 1})
\]
for some rational number $r$ with $r''=[3,3]$.
For this rational number $r$,
since $CS(r')=CT(r)=CT(w_1' \cdots w_6')$ consists of $1$ and $2$,
we have $r'=[1,4,3]$.
Furthermore since $CS(r)=CS(w_1' \cdots w_6')$ consists of $4$ and $5$,
we finally have $r=[4,2,4,3]$ which equals $r_1$ in the statement of the theorem.
This completes the proof of Lemma~\ref{lem:claim2}.
\end{proof}

\begin{lemma}
\label{lem:claim3}
Under the foregoing notation,
$(f(u_{r_i}))=_{G_0} (u_{r_{i+1}}^{\pm 1})$ for every $i \ge 1$.
\end{lemma}

\begin{proof}
Fix $i \ge 1$.
Then $r_i=[4, 2, m_3, \dots, m_k]$ with $m_3 \ge 3$.
By Lemma~\ref{lem:properties}(2),
$CS(r_i)$ consists of $4$ and $5$ without $(5,5)$.
Moreover, since $r_i'=[1, m_3, \dots, m_k]$,
by Lemmas~\ref{lem:properties}(2) and \ref{lem:induction1},
$CT(r_i)=CS(r_i')$ consists of $1$ and $2$,
which implies that the number of occurrences of $4$'s between any two $5$'s is one or two.

\medskip
\noindent{\bf Claim.}
{\it
By cutting the cyclic word $(u_{r_i})$
in the middle of each positive or negative subwords of length $5$,
we may write $(u_{r_i})$ as a product
$(v_{i,1} \cdots v_{i,k_i})$,
where each $v_{i,j}$ is one of the following:
\[
\begin{aligned}
v_1& \equiv \bar{b} abab \bar{a}\bar{b}\bar{a}\bar{b} ab,\\
v_2& \equiv v_1^{-1} \equiv \bar{b}\bar{a} baba \bar{b}\bar{a}\bar{b}\bar{a} b,\\
v_3& \equiv abab \bar{a}\bar{b}\bar{a}\bar{b} abab \bar{a}\bar{b}\bar{a},\\
v_4& \equiv v_3^{-1} \equiv aba \bar{b}\bar{a}\bar{b}\bar{a} baba \bar{b}\bar{a}\bar{b}\bar{a},\\
v_5& \equiv \bar{b} abab \bar{a}\bar{b}\bar{a},\\
v_6& \equiv \bar{b}\bar{a} baba\bar{b}\bar{a}\bar{b}\bar{a},\\
v_7& \equiv v_6^{-1} \equiv abab \bar{a}\bar{b}\bar{a}\bar{b} ab,\\
v_8& \equiv v_5^{-1} \equiv aba \bar{b}\bar{a}\bar{b}\bar{a} b.
\end{aligned}
\]
}

\begin{proof}[Proof of Claim]
Note that for every $n=1, \dots, 8$,
$v_n$ is an alternating word in $\{a,b\}$
such that $S(v_n)=(k_n, t_n\langle 4 \rangle, \ell_n)$,
where $t_n\in \{1,2\}$ and $k_n, \ell_n \in\{1,2,3,4\}$.
Consider the graph as in Figure~\ref{fig.graph},
where the vertex set is equal to $\{v_1, \dots, v_8\}$
and each edge is endowed with one or two orientations.
Observe that if $v_n$ and $v_m$ are the initial and terminal vertices, respectively,
of an oriented edge of the graph, then
the word $v_nv_m$ is an alternating word such that
$S(v_nv_m)=(k_n, t_n\langle 4\rangle, 5, t_m\langle 4\rangle, \ell_m)$,
namely,
the terminal subword of $v_n$,
corresponding to the last component $\ell_n$ of $S(v_n)$, and
the initial subword of $v_m$,
corresponding to the first component $k_m$ of $S(v_m)$,
are amalgamated into
a maximal positive or negative alternating subword of $v_nv_m$,
of length $5$.
Moreover, the weight $t_n$ (resp. $t_m$) is $1$ or $2$
according to whether the vertex $v_n$ (resp. $v_m$)
has valence $3$ or $2$.
Thus, if $v_{n_1}, v_{n_2}, \dots, v_{n_p}$, where $v_{n_j} \in \{v_1, \dots, v_8\}$,
is a closed edge path in the graph which is compatible with
the specified edge orientations
(a compatible closed edge path, in brief),
namely, if
$v_{n_j}$ and $v_{n_{j+1}}$ are the initial and terminal vertices
of an oriented edge of the graph for each
$j=1,2,\dots, p$, where the indices are considered modulo $p$,
then the cyclically reduced word
$v_{n_1} v_{n_2} \cdots v_{n_p}$
is a cyclically alternating word with $CS$-sequence
$\lp 5, t_{n_1}\langle 4\rangle, 5, t_{n_2}\langle 4\rangle,
5, \dots, t_{n_p}\langle 4\rangle \rp$.

\begin{figure}[h]
\begin{center}
\includegraphics{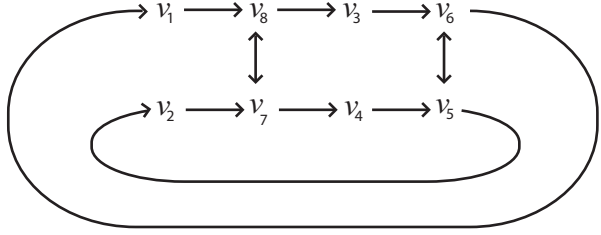}
\end{center}
\caption{\label{fig.graph}
Proof of Claim in the proof of Lemma~\ref{lem:claim3}
}
\end{figure}

Since the weight $t_{n_j}$ is $1$ or $2$
according to whether the vertex $v_{n_j}$ has valence $3$ or $2$,
we see that
for any compatible closed edge path,
the $CT$-sequence $\lp t_{n_1}, t_{n_1}, \dots,t_{n_p} \rp$
of the corresponding cyclically alternating word
consists of only $1$ and $2$ and that it has isolated $2$'s.
Moreover, for any such cyclic sequence, we can construct
a compatible closed edge path
such that the $CT$-sequence of the corresponding cyclically alternating word
is equal to the given cyclic sequence.
In particular, we can find a compatible closed edge path
such that
the $CT$-sequence of the corresponding cyclically alternating word, $(w)$,
is equal to $CT(u_{r_i})$.
This implies that $CS(w)=CS(u_{r_i})$.
Hence
$(w) \equiv (u_{r_i}^{\pm 1})$ as cyclic words
by \cite[Lemma~3.2]{lee_sakuma}.
This completes the proof of Claim.
\end{proof}

Putting
\[
\begin{aligned}
w_1& :\equiv b \bar{X}\bar{b}\bar{X}\bar{b} XbXb \bar{X}\bar{b},\\
w_3& :\equiv\bar{X}\bar{b}\bar{X}\bar{b} XbXb \bar{X}\bar{b}\bar{X} \bar{b} XbX,\\
w_5& :\equiv b \bar{X}\bar{b}\bar{X}\bar{b} XbX,\\
w_6& :\equiv bX \bar{b}\bar{X}\bar{b}\bar{X} bXbX,\\
w_2&:\equiv w_1^{-1}, \quad w_4 :\equiv w_2^{-1}, \quad
w_7 :\equiv w_6^{-1} \quad \textrm{and} \quad w_8 :\equiv w_5^{-1},
\end{aligned}
\]
we obviously have
$f(v_n)=w_n$ for every $n=1,2, \dots, 8$,
so that
\[
(f(u_{r_i}))=(f(v_{i,1} \cdots v_{i,k_i}))=(w_{i,1} \cdots w_{i,k_i}),
\]
where each $w_{i,j} \in \{w_1, \dots, w_8\}$.

Recall from Claims~1 and 2 in the proof of Lemma~\ref{lem:claim2}
that $\bar{X}\bar{b}\bar{X}bXbX\bar{b}=_{G_0} z_1$,
where $z_1 \equiv a \cdots \bar{b}$ is
the alternating word
in $\{a,b\}$ with $S(z_1)=(4,5,4,5)$,
and that
$\bar{X}\bar{b}\bar{X} \bar{b}XbX=_{G_0} z_2$,
where $z_2 \equiv a \cdots \bar{a}$ is
the alternating word
in $\{a,b\}$ with $S(z_2)=(4,4,5,4)$.
It follows that
\[
\begin{aligned}
w_1&=_{G_0} z_1^{-1} b\bar{X}\bar{b} \equiv : w_1',\\
w_3&=_{G_0} z_2 z_1^{-1} \equiv : w_3',\\
w_5&=_{G_0} z_1^{-1} \equiv : w_5',\\
w_6&=_{G_0} bX \bar{b} z_2^{-1} \equiv : w_6',\\
w_2&=w_1^{-1}=_{G_0} (w_1')^{-1} \equiv : w_2',\\
w_4&=w_3^{-1}=_{G_0} (w_3')^{-1} \equiv : w_4'\\
w_7&=w_6^{-1}=_{G_0} (w_6')^{-1} \equiv : w_7'\\
w_8&=w_5^{-1}=_{G_0} (w_5')^{-1} \equiv : w_8'.
\end{aligned}
\]
Then we have
\[
(f(u_{r_i}))=_{G_0}(w_{i,1}' \cdots w_{i,k_i}'),
\]
where each $w_{i,j}' \in \{w_1', \dots, w_8'\}$.
Moreover
\[
\begin{aligned}
w_1' &\equiv b \cdots \bar{b}, \quad w_2' \equiv b \cdots \bar{b}, \quad
w_3' \equiv a \cdots \bar{a}, \quad w_4' \equiv a \cdots \bar{a},\\
w_5' &\equiv b \cdots \bar{a}, \quad w_6' \equiv b \cdots \bar{a}, \quad
w_7' \equiv a \cdots \bar{b}, \quad w_8' \equiv a \cdots \bar{b}
\end{aligned}
\]
are alternating words in $\{a,b\}$ such that
\[
\begin{aligned}
S(w_1')&=(S(z_1^{-1}), S(b\bar{X}\bar{b}))=(5,4,5,4,5,4,5,4,4,5)\\
S(w_3')&=(S(z_2), S(z_1^{-1}))=(4,4,5,4,5,4,5,4),\\
S(w_5')&=S(z_1^{-1})=(5,4,5,4),\\
S(w_6')&=(S(bX \bar{b}), S(z_2^{-1}))=(5,4,4,5,4,5,4,5,4,4),
\end{aligned}
\]
and $S(w_2')=S((w_1')^{-1})=(5,4,4,5,4,5,4,5,4,5)$,
$S(w_4')=S((w_3')^{-1})=(4,5,4,5,4,5,4,4)$,
$S(w_7')=S((w_6')^{-1})=(4,4,5,4,5,4,5,4,4,5)$, and
$S(w_8')=S((w_5')^{-1})=(4,5,4,5)$.

Observe in the graph in Figure~\ref{fig.graph} that
if $v_n$ and $v_m$ are the initial and terminal vertices, respectively,
of an oriented edge,
then $w_n'w_m'$ is an alternating word such that
$S(w_n'w_m')=(S(w_n'), S(w_m'))$,
which consists of $4$ and $5$, and moreover the components $5$ are isolated.
This observation yields that
\[
\begin{aligned}
CS(w_{i,1}' \cdots w_{i,k_i}')&=\lp S(w_{i,1}'), \dots, S(w_{i,k_i}') \rp,\\
CT(w_{i,1}' \cdots w_{i,k_i}')&=\lp T(w_{i,1}'), \dots, T(w_{i,k_i}') \rp.
\end{aligned}
\]
Here
\[
\begin{aligned}
T(w_1')&=(1,1,1,2), \quad T(w_2')=(2,1,1,1), \\
T(w_3')&=(2,1,1,1), \quad T(w_4')=(1,1,1,2), \\
T(w_5')&=(1,1), \quad \qquad T(w_6')=(2,1,1,2), \\
T(w_7')&=(2,1,1,2), \quad T(w_8')=(1,1).
\end{aligned}
\]
This also yields that
\[
CV(w_{i,1}' \cdots w_{i,k_i}')=\lp V(w_{i,1}'), \dots, V(w_{i,k_i}') \rp,
\]
where $V(w_n')=(3)$ if $n=1,2,3,4$, and $V(w_n')=(2)$ otherwise.

Define $N(v_n)$ to be the number of positive or negative proper subwords of $v_n$
of length $4$ for each $n=1, \dots, 8$. Here, by a proper subword of $v_n$,
we mean a subword which lies in the interior of $v_n$.
Then we see that $V(w_n')=(N(v_n)+1)$ for each $n=1, \dots, 8$.
Since $(v_{i,1} \cdots v_{i,k_i})$ is a product
being cut in the middle of each positive or negative subwords of length $5$,
we also see that
\[
\lp N(v_{i,1}), \dots, N(v_{i,k_i}) \rp=CT(r_i)=CS(r_i')
\]
with $r_i'=[1, m_3, \dots, m_k]$.
Since $V(w_{i,j}')=(N(v_{i,j})+1)$
for each $j=1, \dots, k_i$,
$CV(w_{i,1}' \cdots w_{i,k_i}')=\lp N(v_{i,1})+1, \dots, N(v_{i,k_i})+1 \rp$
is the $CS$-sequence corresponding to the rational number
$[2, m_3, \dots, m_k]$.
Hence
\[
(f(u_{r_i}))=_{G_0} (w_{i,1}' \cdots w_{i,k_i}') \equiv (u_{r}^{\pm 1})
\]
for some rational number $r$ with $r''=[2, m_3, \dots, m_k]$.
For this rational number $r$,
since $CS(r')=CT(r)=CT(w_{i,1}' \cdots w_{i,k_i}')$ consists of $1$ and $2$,
we have $r'=[1,3, m_3, \dots, m_k]$.
Furthermore, since $CS(r)=CS(w_{i,1}' \cdots w_{i,k_i}')$
consists of $4$ and $5$,
we finally have $r=[4,2,3,m_3, \dots, m_k]$
which equals $r_{i+1}$ in the statement of the theorem.
This completes the proof of Lemma~\ref{lem:claim3}.
\end{proof}

Since $G=\langle a, b \svert u_{r_0}=u_{r_1}=u_{r_2}=\cdots=1 \rangle$,
Lemmas~\ref{lem:claim1}--\ref{lem:claim3} imply that
$f$ descends to an epimorphism $\hat f: G \rightarrow G$.
Now we show that $\hat f$ is not an isomorphism.
Let $s=[3,3,4]$.
Then
\[CS(u_s)=CS(s)=\lp 3,3,3,4,3,3,4,3,3,4,3,3,4,3,3,3,4,3,3,4,3,3,4,3,3,4 \rp,
\]
so that
\[
\begin{aligned}
(u_s)& \equiv (aba \bar{b}\bar{a}\bar{b} aba \bar{b}\bar{a}\bar{b}\bar{a}
bab \bar{a}\bar{b}\bar{a} baba
\bar{b}\bar{a}\bar{b} aba \bar{b}\bar{a}\bar{b}\bar{a}
bab \bar{a}\bar{b}\bar{a} baba\\
&\qquad \bar{b}\bar{a}\bar{b} aba \bar{b}\bar{a}\bar{b} abab
\bar{a}\bar{b}\bar{a} bab \bar{a}\bar{b}\bar{a}\bar{b}
aba \bar{b}\bar{a}\bar{b} abab
\bar{a}\bar{b}\bar{a} bab \bar{a}\bar{b}\bar{a}\bar{b}).
\end{aligned}
\]
As in the proof of Lemma~\ref{lem:claim1},
letting $w \equiv a \cdots \bar{a}$ be an alternating word
in $\{a,b\}$ such that $S(w)=(3,3,3,4,3,3)$, we have
\[
(u_s) \equiv (w b w baba \bar{b}\bar{a}\bar{b} w^{-1} \bar{b} w^{-1}\bar{b}).
\]

\begin{lemma}
\label{lem:claim4}
We have $\hat f(u_s)=1$.
\end{lemma}

\begin{proof}
Clearly
\[
(f(u_s))= (f(w) \bar{b} f(w) \bar{b}\bar{X}\bar{b}\bar{X} bXb f(w^{-1}) b f(w^{-1}) b).
\]
Here, since $\hat f(w)=\bar{a}$ from the proof of Lemma~\ref{lem:claim1}, we have
\[
(f(u_s))=_G(\bar{a}\bar{b}\bar{a}\bar{b} \bar{X}\bar{b} \bar{X} bX babab),
\]
where $(\bar{a}\bar{b}\bar{a}\bar{b} \bar{X}\bar{b} \bar{X} bX babab)$
is a cyclically alternating word in $\{a,b\}$ such that
\[
\begin{aligned}
CS(\bar{a}\bar{b}\bar{a}\bar{b} \bar{X}\bar{b} \bar{X} bX babab)&=\lp S(\bar{a}\bar{b}\bar{a}\bar{b}), S(\bar{X}\bar{b}), S(\bar{X}), S(bX), S(babab)\rp,\\
&=\lp 4,(4,4,5,4,4,5),(4,4,5,4,4,4),(5,4,4,5,4,4),5 \rp
\end{aligned}
\]
which equals $CS(r_0)$.
This implies that $(f(u_s))=_G(\bar{a}\bar{b}\bar{a}\bar{b} \bar{X}\bar{b} \bar{X} bX babab)=_G 1$,
namely $\hat f(u_s)=1$, as required.
\end{proof}

\begin{lemma}
\label{lem:claim5}
Under the foregoing notation,
let $R$ be the symmetrized subset of
$F(a,b)$ generated by the set of relators
$\{u_{r_i} \svert i \ge 0\}$
of the upper presentation
$G=\langle a, b \svert u_{r_0}=u_{r_1}=u_{r_2}=\cdots=1 \rangle$.
Then $R$ satisfies $C(4)$ and $T(4)$.
\end{lemma}

\begin{proof}
Since every element in $R$ is cyclically alternating,
$R$ clearly satisfies $T(4)$.
To show that $R$ satisfies $C(4)$,
we begin by setting some notation.
Recall from Lemma~\ref{lem:sequence} that
for every rational number $r$ with $0<r\le 1$,
$CS(r)$ has a decomposition
$\lp S_1,S_2,S_1,S_2 \rp$ depending on $r$.
For clarity,
we write $\lp S_1(r),S_2(r),S_1(r),S_2(r) \rp$
for this decomposition.
On the other hand,
if $r$ is a rational number with
$r=[m_1, \dots, m_k]$ with $k \ge 2$ and $(m_2, \dots, m_k) \in (\ZZ_{\ge 2})^{k-1}$,
then the symbol $r^{(n)}$ denotes the rational number
with continued fraction expansion
$[m_{n+1}-1,m_{n+2},\dots, m_k]$
for each $n=1, \dots, k-1$.

\medskip
\noindent{\bf Claim 1.}
{\it
For any two integers $i, j \ge 0$ with $i \neq j$,
the cyclic word $(u_{r_j})$ does not contain
a subword corresponding to $(S_1(r_i))$
or $(\ell_1, S_2(r_i), \ell_2)$ with $\ell_1, \ell_2 \in \ZZ_+$.
}

\begin{proof}[Proof of Claim 1]
Suppose on the contrary that there are some $i \neq j$
such that the cyclic word $(u_{r_j})$ contains
a subword corresponding to $S_1(r_i)$ or $(\ell_1, S_2(r_i), \ell_2)$.
We first show that this assumption implies that
$CS(r_j)$ contains $S_1(r_i)$ or $S_2(r_i)$ as a subsequence.
If $(u_{r_j})$ contains
a subword corresponding to $(\ell_1, S_2(r_i), \ell_2)$,
then clearly $CS(r_j)$ contains $S_2(r_i)$
as a subsequence.
So assume that $(u_{r_j})$ contains
a subword corresponding to $S_1(r_i)$.
Then $CS(r_j)$ contains $(d_1+s_1, s_2, \dots, s_{t-1}, s_t+d_2)$
as a subsequence,
where $S_1(r_i)=(s_1, s_2, \dots, s_t)$.
Since the continued fraction expansions of both $r_i$ and $r_j$
begin with $4$, we see that
$S_1(r_i)$ begins and ends with $5$ by Lemma~\ref{lem:sequence}(3)
and that
$CS(r_j)$ also consists of $4$ and $5$ by Lemma~\ref{lem:properties}(2).
Hence, we must have $d_1=d_2=0$ and therefore
$CS(r_j)$ contains $S_1(r_i)$ as a subsequence.
Thus we have proved that $CS(r_j)$ contains $S_1(r_i)$ or $S_2(r_i)$
as a subsequence.

Note that the lengths of
the continued fraction expansions of $r_i$ and $r_j$
are $i+3$ and $j+3$, respectively.
Hence we can apply Lemma~\ref{lem:useful}
successively to see that
$CS(r_j^{(n)})$ contains
$S_1(r_i^{(n)})$ or $S_2(r_i^{(n)})$
as a subsequence for every $n=1, \dots, \min\{i+1,j+1\}$.
Since $i \neq j$, there are two cases.

\medskip
\noindent{\bf Case 1.} $j > i \ge 0$.
Recall that $r_i$ is equal to $[4,3,3]$ or
$r_i=[4,2,(i-1)\langle 3 \rangle,4,3]$
according to whether $i=0$ or $i \ge 1$.
So we have $r_i^{(i+1)}=[m,3]$.
Here, $m=2$ if $i=0$, and $m=3$ otherwise.
Since $j>i$, we can observe that $r_j^{(i+1)}$ has a continued fraction expansion
of the form $[m-1,n_1,\dots, n_k]$, where $k \ge 2$
and each $n_t$ is $3$ or $4$.
Since $S_1(r_i^{(i+1)})=(m+1)$ and $CS(r_j^{(i+1)})$ consists of $m-1$ and $m$,
the cyclic sequence $CS(r_j^{(i+1)})$ cannot contain
$S_1(r_i^{(i+1)})=(m+1)$ as a subsequence.
Hence $CS(r_j^{(i+1)})$ must contain
$S_2(r_i^{(i+1)})$ as a subsequence.
But since $r_j^{(i+1)}=[m-1,n_1,\dots, n_k]$ with $n_1 \ge 3$,
$(m,m)$ does not occur in $CS(r_j^{(i+1)})$
by Lemma~\ref{lem:properties}(2).
Since $S_2(r_i^{(i+1)})=(m,m)$ by Lemma~\ref{lem:relation}(1),
this implies that $S_2(r_i^{(i+1)})$
cannot occur in $CS(r_j^{(i+1)})$, a contradiction.

\medskip
\noindent{\bf Case 2.} $i > j \ge 0$.
As in Case~1, we can observe that $r_j^{(j+1)}=[m,3]$,
where $m=2$ if $j=0$, and $m=3$ otherwise, and that
$r_i^{(j+1)}$ has a continued fraction expansion
of the form $[m-1,n_1,\dots, n_k]$, where $k \ge 2$
and each $n_t$ is $3$ or $4$.
Then both $S_1(r_i^{(j+1)})$ and $S_2(r_i^{(j+1)})$
contain a term $m-1$ by Lemma~\ref{lem:relation}(2).
But since $CS(r_j^{(j+1)})$ consists of only $m$ and $m+1$,
this is impossible.
\end{proof}

By Claim~1, we see that
the assertion in Lemma~\ref{lem:max_piece} holds
even if
$(u_r^{\pm 1})$ is replaced by $(u_{r_i}^{\pm 1})$ for any $i \ge 0$
and the symmetrized subset $R$ in the lemma is enlarged to be the
set in the current setting, namely,
$R$ is the symmetrized subset of
$F(a,b)$ generated by the set of relators
$\{u_{r_i} \svert i \ge 0\}$
of the group presentation
$G=\langle a, b \svert u_{r_0}=u_{r_1}=u_{r_2}=\cdots=1 \rangle$.
To be precise, the following hold.

\medskip
\noindent{\bf Claim 2.}
{\it
For each $i \ge 0$,
a subword $w$ of the cyclic word
$(u_{r_i}^{\pm 1})$ is a piece with respect to
the symmetrized subset $R$ in Lemma~\ref{lem:claim5}
if and only if
$S(w)$ contains neither $S_1(r_i)$ nor
$(\ell_1, S_2(r_i), \ell_2)$ with $\ell_1,\ell_2\in\ZZ_+$
as a subsequence.
}
\medskip

By using Claim~2, we can see,
as in \cite[Proof of Corollary~5.4]{lee_sakuma},
that each cyclic word
$(u_{r_i}^{\pm 1})$ is not a product of less than $4$ pieces
with respect to $R$.
Hence $R$ satisfies $C(4)$.
\end{proof}

\begin{lemma}
\label{lem:claim6}
Under the foregoing notation,
$u_s \neq_G 1$.
\end{lemma}

\begin{proof}
Suppose on the contrary that $u_s =_G 1$.
Then there is a reduced van Kampen diagram $\Delta$ over
$G=\langle a, b \svert u_{r_0}=u_{r_1}=u_{r_2}=\cdots=1 \rangle$
such that
$(\phi(\partial \Delta)) \equiv (u_s)$
(see \cite{lyndon_schupp}).
Since $\Delta$ is a $[4,4]$-map by Lemma~\ref{lem:claim5},
$(\phi(\partial \Delta))$ contains a subword of some $(u_{r_i}^{\pm 1})$
which is a product of $3$ pieces with respect to
the symmetrized subset $R$ in Lemma~\ref{lem:claim5}
(see \cite[Section~6]{lee_sakuma}).
This implies that $CS(\phi(\partial \Delta))$ must contain a term $5$,
which is a contradiction to the fact $CS(\phi(\partial \Delta))=CS(u_s)=CS(s)$
consists of only $3$ and $4$.
\end{proof}

Lemma~\ref{lem:claim6} together with Lemma~\ref{lem:claim4}
shows that $\hat f$ is an epimorphism of $G$, but not an isomorphism of $G$.
Consequently, $G$ is non-Hopfian,
and the proof of Theorem~\ref{thm:simplest_case} is now completed.
\qed

\bibstyle{plain}

\end{document}